\documentclass[11pt]{amsart}
\usepackage{amsthm}
\usepackage{amsxtra}
\usepackage{amssymb}
\usepackage{bbm}
\usepackage{amsfonts}
\usepackage{textcomp}
\usepackage{lmodern}
\usepackage{dsfont}
\usepackage{hyperref}

\newtheorem{theorem}{Theorem}
\numberwithin{theorem}{section}

\newtheorem{lemma}[theorem]{Lemma}

\newtheorem{corollary}[theorem]{Corollary}

{\theoremstyle{definition}
\newtheorem{definition}[theorem]{Definition}
}

{\theoremstyle{definition}
\newtheorem{remark}[theorem]{Remark}
}

\numberwithin{equation}{section}
\numberwithin{figure}{section}

\renewcommand{\Re}{\mathop{\rm Re}\nolimits}
\renewcommand{\Im}{\mathop{\rm Im}\nolimits}

\newcommand{\R}{\mathbb R}
\newcommand{\C}{\mathbb{C}}

\newcommand{\Z}{\mathbb{Z}}
\renewcommand{\L}{\mathcal{L}}
\newcommand{\Op}{\mathrm{Op}}

\newcommand{\w}{{\rm w}}
\newcommand{\la}{\langle}
\newcommand{\ra}{\rangle}
\newcommand{\hf}{{\frac 12}}
\newcommand{\thf}{{\tfrac 12}}

\newcommand{\bigO}{\mathcal{O}}

\renewcommand{\S}{\mathcal{S}}

\newcommand{\eps}{\epsilon}

\newcommand{\one}{{\mathds 1}}

\DeclareMathOperator{\order}{order}

\begin{document}

\title[Subelliptic Fokker--Planck]{Lindblad evolution with subelliptic diffusion}
\author[H. F. Smith]{Hart F. Smith}
\address{Department of Mathematics, University of Washington, Seattle, WA 98195-4350, USA}
\email{hfsmith@uw.edu}
  

\begin{abstract}
We consider classical/quantum correspondence in Lindblad evolution with jump operators for which the corresponding Fokker--Planck equation 
is subelliptic. This allows us to consider the physical model proposed by Zurek and Paz \cite{zurek}, and to extend some of the recent mathematical results of 
Hernandez, Ranard and Riedel \cite{HRR}, \cite{MR4837215}, Galkowski and Zworski \cite{gaz5}, and Li \cite{Li}, where the diffusion term in the Fokker-Planck equation was assumed elliptic.
We consider the case where the jump operators $\ell_j$ in the Lindbladian are linear functions of $x$, and place an assumption which implies that the H\"ormander condition holds for the resulting Fokker-Planck equation. By constructing a suitable parametrix for this equation we show that the semiclassical derivative estimates established in \cite{gaz5} and \cite{Li} for elliptic diffusion also hold in the subelliptic case, with global bounds in $L^p$ for all $1\le p\le \infty$.
\end{abstract}

\maketitle

\section{Introduction}

In this paper we consider the question of the length of time for which the classical/quantum correspondence is valid for quantum systems coupled to larger environments. The basic model comes from the work of Zurek and Paz \cite{zurek} and we are motivated by recent mathematical work of 
Hernandez, Ranard and Riedel \cite{HRR}, \cite{MR4837215}, Galkowski and Zworski \cite{gaz5}, and Li \cite{Li}.
The paper \cite{HRR} in particular can be consulted for physical intuitions and additional indications to the literature. 

To explain the setting we first consider classical evolution governed by a Hamiltonian $ p = p ( x, \xi ) $. Here we use the mathematical notation in which $ x$ and $ \xi $ are respectively the position and momentum variables; a typical $ p $ is given by $ p ( x , \xi ) = \frac 12|\xi|^2 + V ( x ) $, where $\frac 12 |\xi|^2 $ is the kinetic energy and $ V ( x ) $ the potential. A classical observable (a function of position and momentum) $ a = a ( x, \xi ) $ evolves in time according to the Hamiltonian flow
$ \varphi_t = \exp t H_p $, $ H_p := \partial_\xi p \cdot \partial x - \partial_x p \cdot \partial_\xi$. This means that
\begin{equation}
\label{eq:ate}   
\partial_t a_t ( x, \xi ) = H_p a_t ( x, \xi ) \ \Longleftrightarrow \ 
a_t ( x, \xi ) = \varphi_t^* a ( x, \xi ) := a ( \varphi_t ( x, \xi ) ) . 
\end{equation}
On the classical level interaction with the environment can be modeled by adding Brownian motion in the momentum variables. This corresponds to changing
the evolution equation in \eqref{eq:ate} to a Fokker--Planck equation:
\begin{equation}
\label{eq:ateFP}
\partial_t a_t ( x, \xi ) = H_p a_t ( x, \xi ) +  \epsilon^2 \Delta_\xi a_t ( x, \xi ) .
\end{equation}
The connection to Brownian motion $ W $ in $ \xi $ comes from considering the Langevin equation 
$$ 
( x'(t), \xi'( t ))  = H_p ( x (t), \xi(t) ) + \epsilon \sqrt{2} \dot W ,
$$
so that $ a_t ( x( 0 ) , \xi ( 0 ) ) = \mathbb E ( a ( x ( t ) , \xi ( t ) ) )$. See Drouot \cite{dr} and Ren--Tao \cite{rent} for recent progress on the closely related kinetic Brownian motion. 

The papers \cite{HRR, gaz5,Li}  considered equations like \eqref{eq:ateFP} with $\Delta_{\xi}$ replaced by $\Delta_{x,\xi}$, that is, with comparable rates of diffusion in both $x$ and $\xi$. With this change the equation becomes a heat equation with drift, and more general cases of $p(x,\xi)$ can be handled.
On the other hand, for $ p = \frac 12|\xi|^2+ V ( x ) $ the equation \eqref{eq:ateFP} as written leads to diffusion in $x$ through the term $\xi\cdot\partial_x$ in $H_p$, but at a different time scale than in $\xi$. 
Intuitively, the relation $\dot x=\xi$ implies that diffusion in $\xi$ at scale $\eps t^{1/2}$  leads to diffusion in $x$ at scale $\eps t^{3/2}$. This nonisotropic diffusion is captured in the propagator for \eqref{eq:ateFP} constructed in Theorem \ref{thm:epsleftinverse} of this paper; see \eqref{eqn:K0} and Definition \ref{def:Kgamma} for the definitions of $K_0$ and $K_\gamma$. It is the $\epsilon$ scaling that is key to the main results of this paper, rather than the $t$ scaling.

The condition needed on $p$ to ensure this rate of diffusion in $x$ is the {\em step-two H\"ormander condition}, which is that the collection of Lie brackets $[H_p,\partial_{\xi_j}]$ for $1\le j\le n$ span the tangent space in $x$ at every point. This implies that equation \eqref{eq:ateFP} is {\em subelliptic} in the variables $(t,x,\xi)$; see H\"ormander \cite{Ho} and Rothschild--Stein \cite{RS}.
In this paper we restrict attention to $p=\frac 12|\xi|^2+V(x)$ for the sake of conciseness. The spanning condition is also key to the {\em dispersive property} of the Schr\"odinger equation.

In the quantum framework, the Hamiltonian and the observables are Weyl quantizations (see \cite[Chapter 4]{e-z}) of the classical Hamiltonian and classical observables:
$ P =  p^{\rm{w}}  ( x, h D ) $ and $ A = a^\w (x, h D ) $,  where $ h $ is the effective Planck constant/wave length of the problem. 
The evolution is now governed by the Schr\"odinger equation:
\begin{equation}
\label{eq:ateq}
\partial_t A_t = \tfrac i h [ P, A_t ] \ \Longleftrightarrow \ A_t = e^{ i t P/h } A e^{ - i t P/h } . 
\end{equation}
Classical/quantum correspondence is valid when $ A_t $ is close to $ a_t^{\rm{w}}( x, h D ) $ from \eqref{eq:ate}. That typically holds only for 
$ t \leq C \log ( 1/h )$, the Ehrenfest time -- see  \cite[Sect.11.4]{e-z} and references given therein. 

The evolved operators $ A_t $ can be written as quantizations of classical observables,  $ A_t = a^\w ( t, x, \xi , h ) $. Using the calculus of Weyl operators, see \cite[Theorems 4.12 and 4.18]{e-z}, the differential equation \eqref{eq:ateq} can be formally written, as in \cite{zurek}, as 
\begin{equation}
\label{eq:ateqf'} 
\begin{split}  & \partial_t a ( t, x, \xi , h ) = H_p a ( t, x, \xi , h ) \\
& \ \ \   -
 2 \sum_{ k = 1}^\infty \frac{ (-1)^k h^{ 2k } }{ ( 2 k + 1)! } B( D )^{2k+1} ( p ( x, \xi ) a ( t, y , \eta , h) )|_{x=y, \xi = \eta } ,
 \end{split} 
 \end{equation}
where $ B ( D ) =  \tfrac12 \sigma ( D_x, D_\xi, D_y, D_\eta )$ with $ \sigma $ the symplectic form and $ D = -i \partial $. We should stress that, mathematically, this is typically a divergent asymptotic expansion. Nevertheless, following \cite{zurek} (see also \cite{gob} for an interesting numerical example in this setting), an interaction with
the environment can be introduced as in \eqref{eq:ateFP} by modifying \eqref{eq:ateqf'} to 
\begin{equation}
\label{eq:ateqf} 
\begin{split}  & \partial_t a ( t , x, \xi , h ) = H_p a ( t,  x, \xi , h) + \varepsilon^2  \Delta_\xi a ( t, x, \xi , h ) \\
& \ \ \   -
 2 \sum_{ k = 1}^\infty \frac{ (-1)^k h^{ 2k } }{ ( 2 k + 1)! } B( D )^{2k+1} ( p ( x, \xi ) a ( t, y , \eta , h) )|_{x=y, \xi = \eta } .
 \end{split} 
 \end{equation}
This formal equation can be presented as an evolution equation for operators $ A_t = a^\w ( t , x, h D , h ) $ by noting that (see 
 \cite[(8.15)]{e-z}) 
 \[  ( \partial_{\xi_j} a )^\w  = - \tfrac i h [ x_j , a ] , \]
 where $ x_j $ denotes the multiplication operator by $ x_j $. Hence, a rigorous formulation of \eqref{eq:ateqf} reads as
\begin{align}
\partial_t A_t & = \frac i h [ P , A_t ] - \frac{\gamma}{ 2h } \sum_{j=1}^n [ x_j, [ x_j , A_t ] ]  
\nonumber
\\
 & = \frac i h [ P , A_t ] + \frac{\gamma}{ 2h } \sum_{ j=1}^J L_j [ A, L_j^* ] + [ L_j , A ] L_j^* := \mathcal L A ,
 \label{eq:Lind1}
 \end{align}
 where in this case $J=n$, and
 \[ \gamma := 2 \epsilon^2 /h  , \ \ \  L_j : = x_j . \]
The equation \eqref{eq:Lind1} is the general form of Lindblad evolution with a finite number of jump operators $ L_j $. The {\em superoperator} $ \mathcal L $ 
generates a semigroup which, at least formally, preserves the trace and total positivity. In fact any 
superoperator with these properties has to be of this form (with possibly an infinite set of jump operators) -- see \cite{HRR} and \cite{gaz5} for more on that and for references. The representation in terms of the Lindbladian is close to the point of view of Caldeira and Leggett \cite{cal}, see also \cite{gao}. (We note here that 
as in \cite{gaz5} we look at the observables rather than densities.)

\subsection{Assumptions on $ P $ and $ L_j$}
Similar to \cite{HRR,gaz5,Li}, in this paper we consider a Schr\"odinger type Hamiltonian
$p(x,\xi)=\frac 12 |\xi|^2+V(x)$. We will assume
$V(x)$ is smooth and real-valued function on $\R^n$, and also assume that
\begin{equation}\label{eqn:Vbounds}
|\partial_x^\alpha V(x)|\le 
\begin{cases}
C\,\la x\ra^{2-|\alpha|},& |\alpha|\le 2,\\
C_\alpha,& |\alpha|\ge 2.
\end{cases}
\end{equation}
We use $H_p$ to denote the corresponding Hamiltonian vector field
$$
H_p=\xi\cdot\partial_x-\nabla V(x)\cdot\partial_\xi.
$$
As in \cite{HRR} we consider a family of jump operators $\{L_j\}_{j=1}^J$ corresponding to multiplication by a linear function $\ell_j$, but we assume that each $\ell_j$ depends only on the space variable $x$, 
$$
\ell_j(x)=\sum_{i=1}^n c_{ji}\,x_i,\quad 1\le j\le J, \quad c_{ji}\in\C.
$$

The Lindbladian $\L$ is defined as in \eqref{eq:Lind1}.
As in the preceding discussion, this models the evolution of an observable $A$ developing under the Schr\"odinger system and interacting stochastically with a larger environment through the collection of jump operators $L_j$.
We note that $\{\ell_j,\bar\ell_j\}=0$ so the Lindblad equation has no friction; that is, the term $\mu$ in \cite[(1.6)]{gaz5} vanishes.

Since $\{\ell_j,\bar\ell_j\}=0$ the corresponding Fokker--Planck operator, which gives the leading term in $\L A$ when $P=p^\w(x,hD)$ and $A=a^\w(x,hD)$, becomes
\begin{equation}\label{eqn:Qdef}
Q=H_p
+\gamma\sum_{j=1}^J \Im(\bar\ell_j\nabla_x\ell_j)\cdot\partial_\xi
+\frac{\gamma h}2\sum_{j=1}^J (\nabla_x\ell_j\cdot\partial_\xi)(\nabla_x\bar\ell_j\cdot\partial_\xi).
\end{equation}
We assume that the second order terms are non-degenerate in $\partial_\xi$, which implies that there is a real, non-singular $n\times n$ matrix $B$,
$$
\sum_{j=1}^J \Re\Bigl((\nabla_x \ell_j)\otimes(\nabla_x\bar\ell_j)\Bigr)= BB^T.
$$
We may then write
\begin{equation}\label{eqn:cholesky}
\sum_{j=1}^J(\nabla_x \ell_j\cdot\partial_\xi)(\nabla_x\bar\ell_j\cdot\partial_\xi)=
\sum_{j=1}^n X_j^2,\quad X_j=\sum_{i=1}^n B_{ij}\partial_{\xi_i}.
\end{equation}
While \eqref{eqn:Qdef} is a more general form than \eqref{eq:ateFP}, we can treat the first sum as part of $H_p$ and its analysis adds no complications over that of \eqref{eq:ateFP}.

\subsection{Classical quantum correspondence}
In Section \ref{sec:classquant} we show that certain results from \cite{gaz5} and \cite{Li} concerning long time classical/quantum correspondence hold for the kind of jump operators considered herein. For simplicity we focus on two particularly important cases of their results. The class of observables considered are given by semi-classical symbols based on $L^q$ spaces. As in those papers the point is that the agreement of Lindblad evolution and of Fokker--Planck evolutions holds for significantly longer times (e.g.\ $ h^{-\frac12} $ vs.\ $ \log (1/h) $) than the agreement of the Schr\"odinger and Hamiltonian evolutions. 

\begin{definition}
The symbol class $S^{L^q}_\rho$ denotes symbols $a(x,\xi,h)$ such that, for all multi-indices $\alpha$, there is $C_\alpha$ so that
$$
\|\partial_{x,\xi}^\alpha a\|_{L^q}\le C_\alpha h^{\frac nq-\rho |\alpha|}.
$$
The topology on $S^{L^q}_\rho$ is that induced by the family of seminorms
$$
\|a\|_{N,S^{L^q}_\rho}=\sum_{|\alpha|\le N} h^{\rho |\alpha|-\frac nq}\|\partial_{x,\xi}^\alpha a\|_{L^q}.
$$
\end{definition}

The following theorem extends Theorem 2 of \cite{Li} for $\rho=\frac 12$ and $\gamma>0$ held constant, and shows that the classical and quantum evolution remain close in the trace norm for times of order $t\lesssim h^{-\frac 12}$. Here $\L^1$ denotes the trace norm for operators acting on $L^2(\R^n)$.

\begin{theorem}\label{thm:Li}
Assume that $a\in S^{L^1}_{1/2}$, and let $a(t)=e^{tQ}a$. If $A(t)$ satisfies
$\partial_tA(t)=\L A(t)$ with $A(0)=\Op_h^\w(a)$, there is $C$ depending on $\|a_0\|_{N,S^{L^1}_{1/2}}$ for some $N$, such that
$$
\|A(t)-\Op_h^\w(a(t))\|_{\L^1}\le Cth^{\frac 12},\quad 0\le t<\infty.
$$
\end{theorem}

The paper \cite{gaz5} established long-time correspondence in the Hilbert-Schmidt norm
$\L^2$, and yields analogous estimates when $h^{\frac 13}\lesssim\gamma\lesssim 1$. 
We prove the following extension of the second estimate in (6.2) of Theorem 4 in \cite{gaz5} to the case of the jump operators considered herein. 

\begin{theorem}\label{thm:gaz5}
Assume that $\frac 12\le \rho\le\frac 23$, $a\in S^{L^2}_{\rho}$, and let $a(t)=e^{tQ}a$. Assume also that $\gamma=\gamma(h)$ satisfies $h^{2\rho-1}\le\gamma\le 1$.
Then if $A(t)$ satisfies
$\partial_tA(t)=\L A(t)$ with $A(0)=\Op_h^\w(a)$, there is $C$ depending on $\|a_0\|_{N,S^{L^2}_\rho}$ for some $N$, such that
$$
\|A(t)-\Op_h^\w(a(t))\|_{\L^2}\le Cth^{2-3\rho},\quad 0\le t<\infty.
$$
\end{theorem}

\subsection{The subelliptic estimate}

Theorem \ref{thm:Li} and \ref{thm:gaz5} are proved using the strategies of papers \cite{Li} and \cite{gaz5} together with global-in-time estimates for the Fokker--Planck evolution $\partial_t a(t)=Q a(t)$, with $Q$ given by \eqref{eqn:Qdef}. 

We assume that both $\gamma, h\in (0,1]$, and let $\eps=\sqrt{\gamma h/2}$. If we treat $\eps$ and 
$\gamma$ as the independent variables then the estimates involve $\eps$, but the constants in the estimates are uniform over $\gamma,\eps\in (0,1]$. 
Li \cite{Li} considered the case that $1\le\gamma\le h^{-1}$, but the presence of the linear term in 
$\partial_\xi$ in \eqref{eqn:Qdef}, i.e.~ the first sum in \eqref{eqn:Qdef}, leads us to require bounds on $\gamma$; for simplicity we assume $\gamma\le 1$. If $\ell_j$ is real valued, however, the estimates of Theorem \ref{thm:main} hold for $\gamma\le h^{-1}$, since in that case $\eps\le 1$ and the first sum in \eqref{eqn:Qdef} goes away.

Our main result, analogous to ones in \cite{gaz5} and \cite{Li}, is the following global-in-time derivative estimate, valid for all $N$ and all $1\le p\le\infty$,
\begin{multline}\label{eqn:Gronwall}
\sup_{0\le t<\infty}\sum_{|\alpha|+|\beta|\le N}
\|(\eps\partial_x)^\alpha(\eps\partial_\xi)^\beta a(t)\|_{L^p(\R^{2n})}\\
\le 
C_N\sum_{|\alpha|+|\beta|\le N}\|(\eps\partial_x)^\alpha(\eps\partial_\xi)^\beta a(0)\|_{L^p(\R^{2n})}.
\end{multline}

In Section \ref{sec:proof} we prove a more general result, Theorem \ref{thm:main}, which also captures the smoothing effect of $e^{tQ}$ for $t>0$.

The proof of \eqref{eqn:Gronwall} is a modification of our earlier work \cite{MR4201982} which concerned the kinetic Brownian equation as in \cite{dr}. Unlike the equations considered in that work the Fokker-Planck equation has a distinguished time variable, which we exploit to handle the non-compact setting. Additionally, rather than use a pseudodifferential construction as in \cite{MR4201982}, in this paper
we construct a parametrix  (to any prescribed order) via a non-isotropic heat kernel iteration, analogous to the isotropic iteration in \cite{Li}. This allows us to establish bounds in $L^p$ for all $1\le p\le\infty$.

We briefly outline the parametrix construction here in the case that $B$ is the identity matrix as in \eqref{eq:ateFP}, and first consider $\eps=1$.

We begin by pulling the Fokker-Planck equation back by the flow along the first-order terms in \eqref{eqn:Qdef}. The leading order terms in the pullback of $\partial_t-Q$ are then $\partial_t-(\partial_\xi+t\partial_x)^2$, using the subelliptic notion of order. It is convenient to make a linear change of variables to put this in the form
$$
\partial_t-\thf\xi\cdot\partial_x-(\partial_\xi+\thf t\partial_x)^2,
$$
which is a H\"ormander type operator involving left-invariant vector fields on a step 2 nilpotent Lie group. We will see that the fundamental solution on this Lie group can be written explicitly in terms of the Gaussian heat kernel.

This produces the leading term in a left/right inverse for $\partial_t-Q$, and we then use heat kernel iteration steps to produce a left or right inverse to any desired order. In the iteration step composition of kernels is expressed as convolution in the group product.

Care needs to be taken in that the pullback of $\partial_\xi$ and $\partial_x$ by the flow gives a frame whose second derivatives in $t$ increase unboundedly with $\xi$ due to the $\xi\cdot\partial_x$ term in the flow. To obtain global estimates we therefore can use only the first couple derivatives in $t$, which is possible since we use the heat kernel form of the fundamental solution rather than the pseudodifferential representation used in \cite{MR4201982}.

The case of $0<\eps\le 1$ is handled by appropriately scaling the above parametrix construction, and taking care to note that only positive powers of $\eps$ arise in the various expansions.


\section{The model operator and the case $\eps=1$}\label{sec:modeldomain}

We work with the following frame of vector fields on $\R^{2n+1}$, in the variable $y=(y_0,y',y'')\in\R\times\R^n\times\R^n$,
\begin{itemize}
\item[{$\bullet$}] 
$Y_0=\partial_{0}-\hf \sum_{j=1}^n y_j\partial_{j+n}$

\item[{$\bullet$}]
$Y_j=\partial_{j}+\hf y_0\partial_{j+n}\;$ for $\;1\le j\le n\rule{0pt}{17pt}$,

\item[{$\bullet$}] $Y_j=\partial_{j}\;$ for $\;j\ge n+1\rule{0pt}{17pt}$.
\end{itemize}
These satisfy the commutation relations
$$
[Y_0,Y_j]=Y_{j+n}\quad\text{if}\quad 1\le j\le n,
$$
with all other commutators equal to $0$. The $Y_j$ form a step-2 nilpotent Lie algebra, and
are the left-invariant frame of vector fields associated to the Lie group structure on $\R^{2n+1}$ with product
$$
z\cdot y=(y_0+z_0,y'+z'\!,y''+z''+\thf z_0y'-\thf y_0z').
$$
Observe that $z^{-1}=-z$.
We will use the shorthand notation 
$$
Y'=(Y_1,\ldots, Y_n),\qquad Y''=(Y_{n+1},\ldots, Y_{2n}).
$$
There is an associated dilation structure, given by
$$
\delta_r(y)=\bigl(r^2 y_0,r y',r^3 y''\bigr),
$$
under which $Y_0$, $Y'$, and $Y''$ are respectively of order $2$, $1$, and $3$, where we say an operator $L$ is of order $j$ if
$$
L(f\circ\delta_r)=r^j(Lf)\circ\delta_r\quad\text{for}\;\,r>0.
$$
For a multi-index $\alpha$, we define its order to be
$$
\order(\alpha)=2\alpha_0+\alpha_1+\cdots+\alpha_d+3\alpha_{d+1}+\cdots+3\alpha_{2d}=2\alpha_0+|\alpha'|+3|\alpha''|.
$$
Consider $\alpha\in \Z\times \Z_+^{2n}$, $\beta\in \Z_+^{2n+1}$. That is, we will allow $\alpha_0<0$, but require $\alpha',\alpha''\ge 0$ and $\beta\ge 0$. Then we have
$$
\order(y^\alpha Y^\beta)=\order(y^\alpha\partial_y^\beta)=\order(\beta)-\order(\alpha).
$$
Observe that $y^\alpha Y^\beta$ can be written as a linear combination of terms of the form $y^\theta\partial_y^\gamma$ of the same order, and vice-versa.

The operator 
$$
P_0(y,\partial_y)=\Bigl(Y_0-\sum_{j=1}^nY_j^2\Bigr)
$$
is subelliptic since it satisfies the H\"ormander condition, see \cite{Ho} and \cite{RS}.
By a result of Folland \cite{F}, $P_0(y,\partial_y)$ has a unique fundamental solution that is homogeneous of degree $-2$, which in this case takes a simple form
\begin{equation}\label{eqn:K0}
K_0(y)=G_{y_0}(y')G_{\frac 1{12} y_0^3}(y''),
\end{equation}
where $G=e^{t\Delta}$ is the standard heat kernel on $\R^n$:
$$
G_t(w)=\one_{t>0}\,(4\pi t)^{-\frac n2}\;e^{-\frac{|w|^2}{4t}}.
$$

The solution to the initial value problem $P_0(y,\partial_y)u=0$  for $y_0>0$ with
$u|_{y_0=0}=f$ is given by
$$
u(y)=\int K_0(y_0,y'-z',y''-z''+\thf y_0z')f(z',z'')\,dz'\,dz''.
$$

\begin{definition}\label{def:Kgamma}
For $\gamma\in \Z\times\Z_+^{2n}$, we define $K_\gamma(y)=y^\gamma K_0(y)$. We say that $K(y)$ is a kernel of order $m$, respectively of order $\le m$, if $K$ is a finite linear combination $\sum c_\gamma K_\gamma$ with $\order(\gamma)=-m$, respectively $\order(\gamma)\ge -m$.
\end{definition}

Observe that for each $\alpha\in\Z_+^{2n+1}$ one can write
\begin{equation}\label{eqn:YKdecomp}
\partial_y^\beta K_0(y)=\sum c_{\beta,\gamma}\,K_\gamma(y),
\end{equation}
where $\gamma',\gamma''\ge 0$, $\order(\gamma)=-\order(\beta)$, and $0\ge\gamma_0\ge -3|\beta|$. That is, $\partial_y^\beta K_0$ is a kernel of order equal to $\order(\beta)$. The same holds for $Y^\beta K_0$.

\begin{lemma}\label{lem1}
If $\gamma\in \Z\times\Z_+^{2n}$ then one can write
$$
K_\gamma(y)=\sum_{j,\theta} b_{j,\theta}\,y_0^j\,\partial_y^\theta K_0(y)=\sum_{j,\theta} c_{j,\theta} \,y_0^j\,Y^\theta K_0(y),
$$
where $\theta_0=0$, $|\theta|\le |\gamma|$ and $\order(y_0^j\partial_y^\theta)=\order(y_0^jY^\theta)=\order(y^\gamma)$.
\end{lemma}
\begin{proof}
The result follows by observing that
$$
y''K_0=-\frac 16\, y_0^3\,Y''K_0,\qquad y'K_0=-2y_0\partial_{y'}K_0=(y_0^2Y''-2 y_0 Y')K_0,
$$
and using the fact that commutation between monomials in $y$ and monomials in $Y$ preserves the order.
\end{proof}

\begin{lemma}\label{lem2}
Assume $f(s)\in C([0,\infty))$, and let $F(z_0)=\displaystyle\int_0^{z_0} f(s)\,ds$. Then
$$
\int K_0(z^{-1}y)f(z_0)K_0(z)\,dz=F(y_0)K_0(y).
$$
\end{lemma}
\begin{proof}
One can be verify by calculation that, for each $y_0,z_0$,
$$
\int K_0(z^{-1}y)K_0(z)\,dz'\,dz''=K_0(y).
$$
The statement also follows from the fact that applying $Y_0-\sum_{j=1}^n Y_j^2$ yields $f(y_0)K_0(y)$ on both sides, and both sides vanish for $y_0<0$.
\end{proof}

\begin{corollary}\label{cor:Kcomp}
Suppose that $\order(\gamma)\ge -1$, and $f_\gamma\in C([0,\infty))$. Then one can write
$$
\int K_0(z^{-1}y)f_\gamma(z_0)K_\gamma(z)\,dz=\sum_\theta f_\theta(y_0)K_\theta(y),
$$
where the sum is finite, $\order(\theta)=\order(\gamma)+2$, and $f_\theta\in C([0,\infty))$.
\end{corollary}
\begin{proof}
We use Lemma \ref{lem1} to write the integral as a finite sum with $\theta_0=0$,
$$
\sum_{j,\theta} b_{j,\theta}\int K_0(z^{-1}y)z_0^jf_\gamma(z_0)\partial_z^\theta K_0(z)\,dz
$$
where $j\ge 0$ since $\order(z_0^j\partial_z^\theta)=-\order(\gamma)\le 1$. 
We integrate by parts in $(z',z'')$, and use that
$$
-\partial_{z''}K_0(z^{-1}y)=\partial_{y''}K_0(z^{-1}y),\quad
-\partial_{z'}K_0(z^{-1}y)=(\partial_{y'}-y_0\partial_{y''})K_0(z^{-1}y).
$$
The integral is thus a finite sum
$$
\sum d_{i,j,\theta}\,y_0^i\partial_y^\theta\int K_0(z^{-1}y)z_0^jf_\gamma(z_0) K_0(z)\,dz,
$$
where the sum is now over $\order(y_0^i\partial_y^\theta)=2j-\order(\gamma)$ and $\theta_0=0$.
We apply Lemma \ref{lem2}, and write
$$
\int_0^{y_0} s^jf_\gamma(s)\,ds=y_0^{j+1}\int_0^1 s^j f_\gamma(st)\,ds.
$$ 
The proof is concluded by applying \eqref{eqn:YKdecomp}.
\end{proof}

We will see that in flow coordinates the Fokker-Planck equation is given by a
differential operator in $y$, with coefficients that are smooth in $(y',y'')$, but with only finitely many bounded derivatives in $y_0$. After an appropriate Taylor expansion in $(y',y'')$, the Fokker-Planck operator is then approximated to a prescribed order by a sum
$$
P_0(y,\partial_y)+\sum_{j=1}^N P_j(y,\partial_y),
$$
where $P_j(y,\partial_y)$ is a differential operator in $(\partial_{y'},\partial_{y''})$ of the form
\begin{equation}\label{eqn:Pjform}
P_j(y,\partial_y)=\sum_{\alpha,\beta} f_{\alpha,\beta}(y_0) y^\alpha \partial_y^\beta
\;:\;\beta_0=0,\;\alpha\ge 0,\;
\order(y^\alpha \partial_y^\beta)=2-j,
\end{equation}
with $f_{\alpha,\beta}\in C^0(\R)$.
The functions $f_{\alpha,\beta}(y_0)$ are in fact smooth, but we have uniform control only on $\sup_{|y_0|< T}|f_{\alpha,\beta}(y_0)|$ for each $T$, and not on derivatives in $y_0$. In our application $|\beta|\le 2$, but that is not needed for the following.

\begin{theorem}\label{thm:solution}
Assume given operators $P_j(y,\partial_y)$ of the form \eqref{eqn:Pjform} for $j\ge 1$. Then there are kernels $K_m$ and $R_N$ of the form
$$
K_m(y)=\sum_{\order(\gamma)=m} b_\gamma(y_0)K_\gamma(y),\quad
R_N(y)=\sum_{\order(\gamma)\ge N-1} r_\gamma(y_0)K_\gamma(y),
$$
with $b_\gamma, r_\gamma\in C(\R)$, so that for each $N$
$$
\Bigl(P_0(y,\partial_y)+\sum_{j=1}^N P_j(y,\partial_y)\Bigr)\Bigl(K_0(y)+\sum_{m=1}^N K_m(y)\Bigr)=\delta(y)+R_N(y).
$$
\end{theorem}
\begin{proof}
We recursively solve $P_0K_m=-\sum_{j=1}^m P_j K_{m-j}$ by letting
$$
K_m(y)=-\int K_0(z^{-1}y)\Bigl(\sum_{j=1}^m P_j K_{m-j}\Bigr)(z)\,dz.
$$
By Corollary \ref{cor:Kcomp} this leads to $K_m$ of the stated form, and we then have
$$
R_N=\sum_{\substack{j,m\le N \\ j+m\ge N+1}} P_j K_m
$$
which yields a sum of terms of the stated form.
\end{proof}
The recursion shows that, for any $T<\infty$, if each $f_{\alpha,\beta}\in C([0,T])$ then so are $b_\gamma$ and $r_\gamma$, and one can bound $\sup_{y_0\le T}|b_\gamma(y_0)|$ and $\sup_{y_0\le T}|r_\gamma(y_0)|$ from upper bounds on $\sup_{y_0\le T}|f_{\alpha,\beta}(y_0)|$ and $T$.


\section{The Fokker-Planck equation}

In this section we represent the Fokker-Planck operator as a sum of squares, and use the flow along $X_0$ to map the fundamental solution of Section \ref{sec:modeldomain} to the $(t,x,\xi)$ variables. To keep notation concise, we will make use of $w=(t,x,\xi)$ and $v=(s,z,\zeta)$ to denote variables on $\R^{2n+1}$. In relating the Fokker-Planck operator to the Lie group frame $Y_j$, $\partial_x$ will be comparable to $Y''=\{Y_j\}_{j=n+1}^{2n}$, and $\partial_\xi$ to $Y'=\{Y_j\}_{j=1}^n$. Thus it is natural to write $w=(w_0,w'',w')$, or $t=w_0$, $x=w''$, and $\xi=w'$.

Consider the following frame of vector fields on $\R^{2n+1}$, where the index $j$ runs over $1\le j\le n$, and the $X_j=\sum_{i=1}^nB_{ji}\cdot\partial_{\xi_i}$ are as in \eqref{eqn:cholesky},
\begin{equation*}
\begin{split}
X_0&=\partial_t-\xi\cdot\partial_x+\Bigl(\nabla_x V(x)+
\gamma \sum_{j=1}^J \Im(\bar\ell_j(x)\nabla_x \ell_j)\Bigr)\cdot\partial_\xi,\\
X_j&=\sum_{i=1}^n B_{ij}\partial_{\xi_i},\\
X_{j+n}&=\sum_{i=1}^n B_{ij}\partial_{x_i}.\rule{0pt}{16pt}
\end{split}
\end{equation*}
The choice of $X_{j+n}$ is such that $[X_0,X_j]=X_{j+n}$.

Let $\Phi_t:\R^{2n}\rightarrow \R^{2n}$ denote the flow map along $X_0$ for time $t$.
Since $X_0$ is divergence free, 
$\Phi_t$ preserves the volume form $dx\wedge d\xi$ for each $t$. From the global $C^\infty$ bounds $|\partial_{x,\xi}^\alpha X_0|\le C_\alpha$ for all $|\alpha|\ge 1$, 
it follows that the map $\Phi_t$ is globally bilipschitz on $\R^{2n}$, with Lipschitz constants bounded uniformly over $t\in[-T,T]$ for each $T>0$.
Additionally, for each $|\alpha|\ge 1$ there are uniform bounds 
\begin{equation}\label{eqn:Phibounds}
\sup_{x,\xi}|\partial_{x,\xi}^\alpha \Phi_t(x,\xi)|\le C_\alpha(T)\quad\text{for all}\;\; t\in[-T,T].\end{equation}

Let $\{Z_j\}_{j=1}^{2n}=(\Phi_{-t})_*\{X_j\}_{j=1}^{2n}$ be the the pullback of $\{X_j\}_{j=1}^{2n}$ under $\Phi_t$. Then  $\{Z_j\}_{j=1}^{2n}$ is a smooth frame  on $\R^{2n}$, and if we set
$$
W(t,x,\xi)=B^T\cdot\Bigl[(\nabla^2 V)(\Phi_t(x,\xi))+\gamma \sum_{j=1}^J \Im\bigl(\nabla \bar\ell_j\otimes\nabla \ell_j\bigr)\Bigr]\cdot (B^T)^{-1},
$$
then for $1\le j\le n$,
$$
\partial_t Z_j=Z_{j+n},\qquad \partial_t Z_{j+n}=-\sum_{k=1}^n W_{jk}(t,x,\xi) Z_k.
$$

Although $|\partial_t W(t,x,\xi)|\sim |\xi\cdot(\nabla_x^3 V)(\Phi_t(x,\xi))|$ is not globally bounded unless $V$ is quadratic, we know from the bounds \eqref{eqn:Vbounds} and \eqref{eqn:Phibounds} that
$$
|\partial_x^\beta\partial_\xi^\alpha W(t,x,\xi)|\le C_{\alpha,\beta}(T)\quad\text{for all}\;\; t\in[-T,T].
$$
From the above equation we deduce that for all $\alpha,\beta$, and $1\le j\le n$,
\begin{equation*}
\sup_{k\le 2}|\partial_t^k \partial_x^\beta\partial_\xi^\alpha Z_j|+
\sup_{k\le 1}|\partial_t^k \partial_x^\beta\partial_\xi^\alpha Z_{j+n}|\le C_{\alpha,\beta}(T)\quad\text{for}\;\;t\in[-T,T].
\end{equation*}

\begin{definition}
We define $C^j_{t}C^\infty_{x,\xi}$ to be the space of functions $f$ on $\R^{2n+1}$ such that, for every $T$ and $\alpha\in\Z_+^{2n}$, there are bounds
$$
\sum_{i=0}^j\,\sup_{|t|\le T}\,\sup_{x,\xi}|\partial_t^i\partial_{x,\xi}^\alpha f(t,x,\xi)|\le C_{\alpha}(T).
$$
Similar notation applies when $(x,\xi)$ is replaced by a variable in $\R^m$.
\end{definition}

Consequently, for $1\le j\le n$,
\begin{align*}
Z_j&=(B^T\partial_\xi)_j+t(B^T\partial_x)_j+A_2(t,x,\xi)\cdot(\partial_\xi,\partial_x),
\\
Z_{j+n}&=(B^T\partial_x)_j+A_1(t,x,\xi)\cdot(\partial_\xi,\partial_x),\rule{0pt}{15pt}
\end{align*}
where $A_j$ is a $(2n)^2$ matrix of functions in $C^j_tC^\infty_{x,\xi}$, with
$$
A_j(t,x,\xi)=\mathcal{O}(|t|^j)\quad\text{as}\quad t\rightarrow 0.
$$
Given a point $(\tilde x,\tilde\xi)$, and $y=(y_0,y',y'')\in \R\times\R^n\times\R^n$, we lastly change variables $(t,x,\xi)$ to $y$ by
$$
t=y_0,\quad \xi=\tilde\xi+B y',\quad x=\tilde x+B y''+\tfrac 12 y_0 B y'.
$$
Then
\begin{align*}
\partial_{y_0}-\thf y'\partial_{y''}&=\partial_t  \\ 
\partial_{y'}+\thf y_0\partial_{y''}&= B^T\partial_\xi +t B^T\partial_x\rule{0pt}{12pt}\\
\partial_{y''}&=B^T\partial_x\rule{0pt}{12pt}
\end{align*}
Finally we replace $(\tilde x,\tilde\xi)$ by $(x,\xi)$ to conclude the following.

\begin{lemma}\label{lem:Xjapprox}
For $w=(t,x,\xi)\in\R^{2n}$ define
$$
\exp_w(y)=\bigl(t+y_0,\Phi_{y_0}(x+B y''+\tfrac 12 y_0 B y',\xi+B y')\bigr).
$$
Then with $Y_j$ the left invariant vector fields defined above, the pullback of $X_j$ to the $y$ coordinates satisfies
$$
X_0=Y_0,
\quad X'=Y'+A_2(y)\cdot(\partial_{y'},\partial_{y''}),
\quad X''=Y''+A_1(y)\cdot(\partial_{y'},\partial_{y''}),
$$
where the $A_j(y)$ are $n\times 2n$ matrix functions with coefficients in $C^j_{y_0}C^\infty_{y',y'',w}$,
and $\partial_{y_0}^iA_j(0,y',y'')=0$ for $0\le i<j$. 
\end{lemma}

\begin{remark}
$A_j(y)$ depends on $(x,\xi)$, although not $t$. 
Differentiating $\exp_w(y)$ in $(x,\xi)$ is equivalent to differentiating in $(y',y'')$, so we have the same bounds for derivatives of $A_j$ in $(x,\xi)$.

\begin{definition}\label{def:thetavw}
Let $\Theta_w(v)$ be the inverse to $\exp_w(y)$:
$$
y=\Theta_w(v)\quad\Leftrightarrow\quad \exp_w(y)=v.
$$
\end{definition}
With $w=(t,x,\xi)$ and $v=(s,z,\zeta)$, this equates to
$$
y_0=s-t,\quad (B y',B y''+\tfrac 12 (s-t)B y')=\Phi_{t-s}(z,\zeta)-(x,\xi).
$$
Thus if we let $\Phi_{s-t}(z,\zeta)=(z_{t-s},\zeta_{t-s})$, then
\begin{equation*}
\Theta_w(v)=\bigl(s-t, B^{-1}(\xi-\zeta_{t-s}),B^{-1}(x-z_{t-s}+\thf(t-s)(\xi-\zeta_{t-s}))\bigr)
\end{equation*}
Lemma \ref{lem:Xjapprox} states that, with $X_j$ acting in $v$ for fixed $w$,
\begin{align*}
X'\bigl(f(\Theta_w(v)\bigr)&=\bigl(Y'f+A_2(y)\cdot(\partial_{y'}f,\partial_{y''}f)\bigr)(\Theta_w(v)),\\
X''\bigl(f(\Theta_w(\cdot)\bigr)&=\bigl(Y''f+A_1(y)\cdot(\partial_{y'}f,\partial_{y''}f)\bigr)(\Theta_w(\cdot)).
\end{align*}
With $X_j$ acting in $w$ for fixed $v$, we have the more direct connection
\begin{equation}\label{eqn:XequalsY}
X_j\bigl(f(-\Theta_w(v)\bigr)=(Y_jf)(-\Theta_w(v)),\quad j\ge 1.
\end{equation}

Since $A_2(0,y',y'')=\partial_{y_0}A_2(0,y',y'')=0$, we can write
$A_2(y)=y_0^2 A_0(y)$, and similarly $A_1(y)=y_0 A_0(y)$, in each case with $A_0\in C^0_{y_0}C^\infty_{y',y'',w}$.
\end{remark}

In exponential coordinates at a given point $w$, for $1\le i\le 2n$  we take a Taylor expansion to order $N$ in $(y',y'')$ about $y'=y''=0$ to write
\begin{equation}\label{eqn:Xjexpansion}
X_i=Y_i+\sum_{k=1}^{2n}\biggl(\sum_{|\alpha|\le N+2} c_{i,\alpha,k}(y_0)\,y^\alpha\partial_k
+\sum_{|\alpha|=N+3}r_{i,\alpha,k}(y)\,y^\alpha\partial_k\biggr)
\end{equation}
Here $\alpha_0=2$ if $1\le i\le n$, and $\alpha_0=1$ if $n+1\le i\le 2n$, so 
$\order(y^\alpha\partial_k)\le \order(Y_i)-2$ in the sums. 
The functions $c_{i,\alpha,k}(y_0)\in C^0_{y_0}C^\infty_w$ are continuous in $y_0$. They depend smoothly on $w$, but satisfy uniform bounds on $[-T,T]$ for each $T$.
Also, $r_{j,\alpha,k}(y)\in C^0_{y_0}C^\infty_{y',y'',w}$. 
Note that $\order(y^\alpha\partial_k)\le -N$ if $|\alpha|\ge N+3$.

Expressing $\partial_k$ for $k\ge 1$ in terms of $Y_k$ with $k\ge 1$ yields an expansion with similar $c_{i,\alpha,k}(y_0)$ and $r_{j,\alpha,k}(y)$ and $\order(y^\alpha Y_k)\le \order(Y_i)-2$,
$$
X_i=Y_i+\sum_{k=1}^{2n}\biggl(\sum_{|\alpha|\le N+2} c_{i,\alpha,k}(y_0)\,y^\alpha Y_k
+\sum_{|\alpha|=N+3}r_{i,\alpha,k}(y)\,y^\alpha\partial_k\biggr).
$$
This relation is symmetric, so we have the reverse expansion with different $c_{i,\alpha,k}\in C^0C^\infty_w$ and $r_{i,\alpha,k}\in C^0_{y_0}C^\infty_{y',y'',w}$, and $\order(y^\alpha X_k)\le \order(Y_i)-2$,
\begin{equation}\label{eqn:Yjexpansion}
Y_i=X_i+\sum_{k=1}^{2n}\biggl(\sum_{|\alpha|\le N+2} c_{i,\alpha,k}(y_0)\,y^\alpha X_k
+\sum_{|\alpha|=N+3}r_{i,\alpha,k}(y)\,y^\alpha\partial_k\biggr).
\end{equation}

Let $P_j$ be the terms of order $2-j$ in the expansion of $X_0-\sum_{i=1}^n X_i^2$ using \eqref{eqn:Xjexpansion}, where $P_0=Y_0-\sum_{j=1}^n Y_j^2$.
Following Theorem \ref{thm:solution}  we can iteratively construct $K_m(w,y)$ of the form
$$
K_m(w,y)=\sum_{\order(\gamma)=m}b_\gamma(y_0,w)\, K_\gamma(y),
$$
so that, with $X_j$ written as a differential operator in $y$,
$$
\Bigl(X_0-\sum_{i=1}^n X_i^2\Bigr)\Bigl(K_0(y)+\sum_{m=1}^N K_m(w,y)\Bigr)=\delta(y)+R_N(w,y),
$$
where $\order(R_N)\le 1-N$, and $b_\gamma\in C^0_{y_0}C^\infty_w$ and $r_\gamma(y_0,y,w)\in C^0_{y_0}C^\infty_{y',y'',w}$. Both are independent of $t$.

We first use this to produce a right inverse for $X_0-\sum_{i=1}^n X_i^2$ modulo an operator of any given negative order.

Given $N$, let $K(v,y)=\sum_{m=0}^N K_m(v,y)$, and define
$$
T_Kf(w)=c_B\int K(v,\Theta_v(w))\,f(v)\,\chi(\Theta_v(w))\,dv,
$$
where $c_B^{-1}=|D_y\exp_v(y)|_{y=0}=\det(B)^2$ and $\chi(y)\in C_c^\infty(\R^{2n+1})$ with 
$\chi(y)=1$ for $|y|<\frac 12$ and $\chi(y)=0$ for $|y|>1$.

Then, with the $X_i$ acting in $w=(t,x,\xi)=\exp_v(y)$, we can write
\begin{equation}\label{eqn:rightinverse}
\Bigl(X_0-\sum_{i=1}^n X_i^2\Bigr)T_Kf(w)=f(w)+\int R_N(v,\Theta_v(w))f(v)\,dv,
\end{equation}
Here $R_N(v,y)$ is a time-independent finite sum of the form
\begin{equation}\label{eqn:RNform}
R_N(v,y)=\!\!\sum_{\order(\gamma)\ge N-1} r_\gamma(v,y)K_\gamma(y),\;\; r_\gamma(v,y)\in C^0_{y_0}C^\infty_{y',y'',v}.
\end{equation}
We will establish smoothing estimates for such $R_N$ in Section \ref{sec:opbounds}.

By hypoellipticity of $X_0-\sum_{i=1}^n X_i^2$ and its transpose, a right parametrix is necessarily also a left parametrix (with care taken concerning the order of the remainder), but we can easily modify the above to directly produce a left parametrix.

For each $N$ and $w$ we want to produce a kernel $K(w,\cdot)$ such that
\begin{multline*}
\int K(w,-y)\Bigl(X_0 f-\sum_{i=1}^n X_i^2 f\Bigr)(\exp_w(y))\,\chi(y)\,dy
\\=
f(w)+\int R_N(w,-y)f(\exp_w(y))\,dy.
\end{multline*}
If we expand $X_0-\sum_{i=1}^n X_i^2$ to order $N$ as above,
we seek kernels $K_m(y)$ of order $-m$ as in 
Theorem \ref{thm:solution} so that, with $P_j^t$ the transpose of $P_j$ with respect to $dy$,
$$
\Bigl(P_0^t(-y,-\partial_y)+\sum_{j=1}^N P_j^t(-y,-\partial_y)\Bigr)\Bigl(K_0(y)+\sum_{m=1}^N K_m(y)\Bigr)=\delta(y)+R_N(y).
$$
Observe that
$$
P_0^t(-y,-\partial_y)=
\Bigl(\partial_0+\thf\sum_{i=1}^n y_i\partial_{i+n}\Bigl)-
\sum_{i=1}^n\Bigl(\partial_i-\thf y_0\partial_{i+n}\Bigr)^2,
$$
which is the conjugation of $P_0(y,\partial_y)$ under $y''\rightarrow -y''$. This conjugation preserves both $K_0(y)$ and $\delta(y)$, hence $P_0^t(-y,-\partial_y)K_0(y)=\delta(y)$. 
This conjugation also preserves the order of each $P_j$. Thus if we set
$$
\tilde K_m(y)=K_m(y_0,y',-y''),
$$
then with $\tilde P_j^t$ the conjugation of $P_j^t$ under $y''\rightarrow -y''$, we require that
$$
\Bigl(P_0(y,\partial_y)+\sum_{j=1}^N \tilde P_j^t(-y,-\partial_y)\Bigr)\Bigl(K_0(y)+\sum_{m=1}^N \tilde K_m(y)\Bigr)=\delta(y)+\tilde R_N(y).
$$
We can iteratively construct such kernels following Theorem \ref{thm:solution}, noting that $\tilde K_\gamma(y)=(-1)^{|\alpha''|}K_\gamma(y)$. We summarize the result in the following.

\begin{theorem}
For each $m\ge 1$ there are a finite collection of $t$-independent functions $b_\gamma(w,y_0)\in C_{y_0}C^\infty_w(\R^{2n+2})$ where 
$\order(\gamma)=m$, such that
$$
K(w,y)=K_0(y)+\sum_{m=1}^N\sum_{\order(\gamma)=m}b_\gamma(w,y_0)\,K_\gamma(y),
$$
satisfies
\begin{multline*}
\int K(w,-y)\Bigl(X_0 f-\sum_{i=1}^n X_i^2 f\Bigr)(\exp_w(y))\,\chi(y)\,dy
\\=
f(w)+\int R_N(w,-y)f(\exp_w(y))\,dy,
\end{multline*}
where $R_N(w,y)$ is of the form \eqref{eqn:RNform}.
\end{theorem}


\section{The case $0<\eps\le 1$}

In this section we construct parametrices for $X_0-\eps^2\sum_{i=1}^n X_i^2$, $\eps\in(0,1]$. 
Consider the scaled kernels
$$
K_{\gamma,\eps}(y)=\eps^{-2n}K_\gamma(y_0,\eps^{-1}y',\eps^{-1}y'').
$$
The scaling is such that $K_{0,\eps}(y_0,\cdot)$ converges to $\delta(y',y'')$ as $y_0\rightarrow 0+$, and
$$
\Bigl(Y_0-\eps^2\sum_{j=1}^n Y_j^2\Bigr)K_{0,\eps}(y)=\delta(y).
$$
More generally,
\begin{equation}\label{eqn:epsgain}
y^\alpha(\partial_{y_0},\eps\partial_{y'},\eps\partial_{y''})^\beta K_{\gamma,\eps}(y)=
\eps^{|\alpha'|+|\alpha''|}\bigl(y^\alpha\partial_y^\beta K_\gamma\bigr)_\eps(y).
\end{equation}
We note that, with the notation of Lemma \ref{lem2}
$$
\int K_{0,\eps}(z^{-1}y)f(z_0)K_{0,\eps}(z)\,dz=F(y_0)K_{0,\eps}(y),
$$
which can be seen by a change of variables and observing that $\eps$-scaling intertwines with group multiplication.

The expansion of $X_0-\eps^2\sum_{j=1}^n X_j^2$ in exponential coordinates $y$ is
$$
\biggl(Y_0-\eps^2\sum_{i=1}^n Y_i^2\biggr)+\sum_{j\ge 1}P_j(\eps,y,\partial_y)
$$
where $P_j$ is a sum of terms with $\order(y^\alpha\partial_y^\beta)=2-j$ of the form
\begin{equation}\label{eqn:Pjepsform}
P_j(\eps,y,\partial_y)=\sum_{|\beta|\le 2}\sum_{\alpha>0} \eps^{2-|\beta|}f_{\alpha,\beta}(y_0)y^\alpha 
(\eps\partial_{y'})^{\beta'}(\eps\partial_{y''})^{\beta''}
\end{equation}
Here $f_{\alpha,\beta}(y_0)\in C^0_{y_0}C^\infty_{w}$. 
The remainder term of order $1-N$ is of the same form with coefficients $f_{\alpha,\beta}(y)\in C^0_{y_0}C^\infty_{y',y'',w}$.

By \eqref{eqn:epsgain}, the iterative procedure of Theorem \ref{thm:solution} yields the following. 
\begin{theorem}
Suppose given differential operators $P_j(\eps,y,\partial_y)$ of the form \eqref{eqn:Pjepsform}.
Then there are kernels $K_m(\eps,y)$, each a finite sum of the form
\begin{equation}\label{eqn:Kmepsform}
K_m(\eps,y)=\sum_{j\ge 0}\sum_{\order(\gamma)=m}b_{j,\gamma}(y_0)\,\eps^{j}K_{\gamma,\eps}(y)
\end{equation}
so that
\begin{multline*}
\Bigl(Y_0-\eps^2\sum_{j=1}^n Y_j^2+\sum_{j=1}^N P_j(\eps,y,\partial_y)\Bigr)\Bigl(K_{0,\eps}(y)+\sum_{m=1}^N K_m(\eps,y)\Bigr)
\\
=\delta(y)+R_N(\eps,y),
\end{multline*}
where $R_N(\eps,y)$ is a finite sum of the form \eqref{eqn:Kmepsform} with $\order(\gamma)\ge N-1$.
\end{theorem}

We can use this to give left and right parametrices for $X_0-\eps^2\sum_{i=1}^n X_i^2$. 
We state the result for the left parametrix, observing that the transpose of an operator $P_j$ of the form \eqref{eqn:Pjepsform} is of the same form.

\begin{theorem}\label{thm:epsleftinverse}
For each $m\ge 1$ there are a finite collection of $t$-independent functions $b_{j,\gamma}(w,y_0)\in C^0_{y_0}C^\infty_w$ where 
$\order(\gamma)=m$, such that
$$
K(\eps,w,y)=K_{0,\eps}(y)+\sum_{m=1}^N\sum_{\order(\gamma)=m} b_{j,\gamma}(w,y_0)\,\eps^j K_{\gamma,\eps}(y),
$$
satisfies
\begin{multline*}
\int K(\eps,w,-y)\Bigl(X_0 f-\eps^2\sum_{i=1}^n X_i^2 f\Bigr)(\exp_w(y))\,\chi(y)\,dy
\\=
f(w)+\int R_N(\eps,w,-y)f(\exp_w(y))\,dy,
\end{multline*}
where $R_N(\eps,w,y)$ is a finite sum over $j\ge 0$ and $\order(\gamma)\ge N-1$,
$$
R_N(\eps,w,y)=\sum r_{j,\gamma}(w,y)\,\eps^{j}K_{\gamma,\eps}(y),
$$
with $r_{j,\gamma}\in C^0_{y_0}C^\infty_{y',y'',w}$.
\end{theorem}


\section{Operator bounds}\label{sec:opbounds}

We start by observing the following, which holds by a change of variables,
\begin{equation}\label{eqn:intbound}
\int \eps^{-2n}\bigl|K_{\gamma,\eps}(y)\bigr|\,dy'\,dy''
= c_\gamma \, y_0^{\frac 12\order(\gamma)}.
\end{equation}
In particular, if $\order(\gamma)>0$ then $\lim_{y_0\rightarrow 0}\|K_{\gamma,\eps}(y_0,\cdot)\|_{L^1}=0$.

We use this to establish fixed time $L^p$ bounds on kernels of nonpositive order, for all $1\le p\le \infty$. The left parametrix is an integral kernel $K(\eps,w,v)$ where $K$ is a finite sum with $\order(\gamma)\ge 0$, and
$$
K(\eps,w,\exp_w(y))=\sum b_{\gamma,\eps}(w,y)K_{\gamma,\eps}(-y)
$$ 
with
$b_{\gamma,\eps}\in C^0_{y_0}C^\infty_{y',y'',w}$.

Recall from Definition \ref{def:thetavw} that $y=\Theta_w(v)$ is the solution to $v=\exp_w(y)$,
and we use notation $v=(s,z,\zeta)$, $w=(t,x,\xi)$. For fixed $w$, respectively fixed $v$, the maps $y\rightarrow v$ and $y\rightarrow w$ are 1-1 and the Jacobian factors satisfy
$$
\biggl|\frac {D\Theta_w(v)}{Dv}\biggr|=\det(B)^{-2}=\biggl|\frac {D\Theta_w(v)}{Dw}\biggr|
$$
From \eqref{eqn:intbound}, and the fact that $y_0=v_0-w_0=s-t$, we deduce that
\begin{lemma}\label{lem:operbound}
For each $\gamma$ there is $C_\gamma$ so that
\begin{align*}
\int \bigl|K_{\gamma,\eps}(-\Theta_w(v))\bigr|\,\delta(v_0)\,dv&\le C_\gamma\,|w_0|^{\frac 12\order(\gamma)},\\
\int \bigl|K_{\gamma,\eps}(-\Theta_w(v))\bigr|\,\delta(w_0)\,dw&\le C_\gamma\,|v_0|^{\frac 12\order(\gamma)}.
\end{align*}
\end{lemma}

We next observe that for $\alpha\in \Z_+^{2n}$, by \eqref{eqn:XequalsY} one can express
$$
(\eps \partial_\xi)^{\alpha'}(\eps \partial_x)^{\alpha''} K_{\gamma,\eps}(-\Theta_w(v))=\sum_\theta c_\theta K_{\theta,\eps}(-\Theta_w(v))
$$
where the finite sum is over $\order(\theta)=\order(\gamma)-\order(\alpha)$.

By the Schur test, together with time translation invariance of $\Theta_w(v)$,
we obtain the following fixed time mapping properties, where
for a function $f\in L^1_{\mathit{loc}}(\R^{2n})$ we use the notation
$$
(\delta\otimes f)(v)=\delta(v_0)\,f(v',v'').
$$

\begin{corollary}\label{cor:operbound}
For all $\gamma$, and all $\alpha\in\Z_+^{2n}$, there is $C_{\alpha,\gamma}(T)$ so that, for all $1\le p\le\infty$ and all $0<\eps\le 1$, and $t=w_0\in [0,T]$,
$$
T_{\gamma,\eps}f(w)=\int K_{\gamma,\eps}(-\Theta_w(v))\,(\delta\otimes f)(v)\,dv
$$
satisfies
$$
\|(\eps \partial_\xi)^{\alpha'}(\eps \partial_x)^{\alpha''}(T_{\gamma,\eps}f)(t,\cdot)\|_{L^p(\R^{2n})}\le C_{\alpha,\gamma}(T)\,t^{\frac 12\order(\gamma)-\frac 12 \order(\alpha)}\|f\|_{L^p(\R^{2n})}.
$$
\end{corollary}

For fixed-time derivative estimates of integer order, we need to work with multiples of order 3, since this is the order of both $\partial_x$ and $(\partial_\xi)^3$. The key result (for $\eps=1$) is the following.

\begin{lemma}
Suppose that $P(y,\partial_{y'},\partial_{y''})$ is a polynomial differential operator of order at most $3m+\order(\gamma)$, where $m\in\Z_+$ and $\order(\gamma)\ge 0$. Then one can write
$$
\bigl(P(y,\partial_{y'},\partial_{y''})K_\gamma\bigr)(-y)=\sum_{\beta,\theta}c_{\beta,\theta}\,(Y^\beta)^t\bigl( K_\theta(-y)\bigr),
$$
where $\order(\beta)\le 3m$, $\beta_0=0$, and $\order(\theta)\ge 0$.
\end{lemma}
\begin{proof}
Observe that $(Y')^t f(-y)=(\partial_{y'}f-\thf y_0\partial_{y''}f)(-y)$, and $(Y'')^tf(-y)=(\partial_{y''}f)(-y)$. 
Writing $\partial_{y'}$ and $\partial_{y''}$ in terms of $(Y')^t$ and $(Y'')^t$ expands $P(y,\partial_{y'},\partial_{y''})K_\gamma$ as a sum, with $|\alpha'|+3|\alpha''|\le 3m+\order(\sigma)$ and $\sigma\ge \gamma$,
$$
\sum c_{\alpha,\sigma}\,(\partial_{y'}-\thf y_0\partial_{y''})^{\alpha'} (\partial_{y''})^{\alpha''}K_{\sigma}.
$$
For a term with $|\alpha''|\ge m$, one can factor out $m$ powers of $\partial_{y''}$ and let the remaining derivatives fall on $K_{\sigma}$, which yields a sum of kernels of the form $K_\theta$ with $\order(\theta)\ge 0$. 

A term with $|\alpha'|+3|\alpha''|\le 3m$ is already in the desired form. For a term with $|\alpha''|<m$ and $|\alpha'|+3|\alpha''|> 3m$ we factor out $(\partial_{y''})^{\alpha''}$ and 
$3m-3|\alpha''|$ powers of $(\partial_{y'}-\thf y_0\partial_{y''})$; the remaining powers applied to $K_{\sigma}$ leads to kernels $K_\theta$ with $\order(\theta)\ge 0$.
\end{proof}

By a dilation in $\eps$, together with \eqref{eqn:Yjexpansion} for $N$ sufficiently large, we obtain the following.

\begin{corollary}
Suppose that  $P(\eps,y,\eps\partial_{y'},\eps\partial_{y''})$ is polynomial in each term, with $\order(P)\le 3m+\order(\gamma)$ where $m\in\Z_+$. Then one can write
$$
\bigl(P(\eps,y,\eps\partial_{y'},\eps\partial_{y''})K_{\gamma,\eps}\bigr)(-y)=\sum_{j,\beta,\theta}c_{j,\beta,\theta}\,\eps^j(\eps^{|\beta|} Y^\beta)^t\bigl( K_{\theta,\eps}(-y)\bigr),
$$
where $j\ge 0$, $\order(\beta)\le 3m$, $\beta_0=0$, and $\order(\theta)\ge 0$. 

Furthermore, with $X_j$ the representation of $X_j$ in exponential coordinates at a given point $w$, one can write
\begin{multline}\label{eqn:Xjtranspose}
\bigl(P(\eps,y,\eps\partial_{y'},\eps\partial_{y''})K_{\gamma,\eps}\bigr)(-y)=\sum_{j,\beta,\theta}c_{j,\beta,\theta}(w,y_0)\,\eps^j(\eps^{|\beta|} X^\beta)^t\bigl( K_{\theta,\eps}(-y)\bigr)\\
+\sum_{j,\sigma} r_{j,\sigma}(w,y)\eps^jK_{\sigma,\eps}(-y)
\end{multline}
for functions $c_{j,\beta,\gamma}\in C^0_{y_0}C^\infty_w$, with the same conditions on $j$, $\beta$ and $\theta$, and $r_{j,\sigma}\in C^0_{y_0}C^\infty_{w,y',y''}$, with $j\ge 0$ and $\order(\sigma)\ge 0$.
\end{corollary}

\begin{theorem}\label{thm:opbound}
Suppose that $\order(\gamma)\ge 0$, and let
$$
T_{\gamma,\eps}f(w)=\int K_{\gamma,\eps}(-\Theta_w(v))\,(\delta\otimes f)(v)\,\chi(\Theta_w(v))\,dv
$$
Then if $|\alpha'|+3|\alpha''|\le 3m+\order(\gamma)$, $m\in\Z_+$, there is $C_{\gamma,m}$ so that, for all $1\le p\le\infty$ and all $0<\eps\le 1$,
\begin{multline*}
\sup_{t\ge 0}\,\bigl\|(\eps\partial_\xi)^{\alpha'}(\eps\partial_x)^{\alpha''}T_{\gamma,\eps}f(t,\cdot)\bigr\|_{L^p(\R^{2n})}\\
\le
C_{\gamma,m}
\sum_{|\beta'|+3|\beta''|\le 3m}\|(\eps\partial_\xi)^{\beta'}(\eps\partial_x)^{\beta''}f\|_{L^p}.
\end{multline*}
\end{theorem}
\begin{proof}
Applying \eqref{eqn:XequalsY} we write
$$
(\eps X')^{\alpha'}(\eps X'')^{\alpha''}K_{\gamma,\eps}(-\Theta_w(v))
=\bigl((\eps Y')^{\alpha'}(\eps Y'')^{\alpha''}K_{\gamma,\eps}\bigr)(-\Theta_w(v)).
$$
We can apply \eqref{eqn:Xjtranspose} since
the transpose in $dv$ is the same as in $dy$,
\begin{multline*}
(\eps X')^{\alpha'}(\eps X'')^{\alpha''}T_{\gamma,\eps}f(w)\\
=\sum c_{j,\theta,\beta}(w,w_0)\,\eps^j\int K_{\theta,\eps}(-\Theta_w(v))\,(\delta\otimes (\eps X)^\beta f)(v)\,\chi(\Theta_w(v))\,dv\\
+\int R(\eps,w,v,-\Theta_w(v))\,(\delta\otimes f)(v)\,dv
\end{multline*}
where the sum is over $\order(\beta)\le 3m$, $\beta_0=0$, and $\order(\theta)\ge 0$. $R$ is a sum of terms $\eps^jr_{j,\sigma}(w,\Theta_w(v))K_{\sigma,\eps}(-\Theta_w(v))$ with $\order(\sigma)\ge 0$. The bound of the theorem now follows from Corollary \ref{cor:operbound}, since we are restricted to a unit time interval by the support of $\chi$.
\end{proof}

\section{Proof of the subelliptic estimate}\label{sec:proof}

We establish \eqref{eqn:Gronwall} as a result of Theorem \eqref{thm:main} below, which counts the order of derivatives in $x$ and $\xi$ as they are counted in the subelliptic calculus.

Let us first verify that \eqref{eqn:Gronwall} holds over $0\le t\le 1$. For $N=0$ this follows from the fact that $e^{tQ}$ is a contraction on $L^p$ for $1\le p\le\infty$ (see \eqref{eqn:Lpbound} below). Since the second order terms in $Q$ are constant coefficient, and all derivatives of $V$ and $\ell_j(x)$ are globally bounded, it follows that $[Q,(\eps\partial_x)^\alpha(\eps\partial_\xi)^\beta]$ is a differential operator in $\eps\partial_x$ and $\eps\partial_\xi$ of order at most $|\alpha|+|\beta|$, with smooth bounded coefficients.
An induction argument using the Duhamel formula and Gronwall's inequality then establishes \eqref{eqn:Gronwall} uniformly over $t\in [0,1]$.

Combining this with \eqref{eqn:globalest} below for $T=1$ now shows that \eqref{eqn:Gronwall} holds globally in time.

\begin{theorem}\label{thm:main}
Assume that $p=\frac 12|\xi|^2+V(x)$ with real $V$ satisfying \eqref{eqn:Vbounds}, and $Q$ is of the form \eqref{eqn:Qdef}, where \eqref{eqn:cholesky} holds with non-singular $B$.
Then for all $N\in\Z_+$ there is $C_N$  such that, with $a(t)=e^{tQ}a(0)$ and
$\eps=\sqrt{\gamma h/2}$, for all $1\le p\le\infty$, and all $\eps,\gamma\in(0,1]$,
\begin{multline}\label{eqn:localest}
\sup_{0\le t\le \infty}\,\sum_{|\alpha|+3|\beta|\le 3N}
\|(\eps\partial_\xi)^\alpha(\eps\partial_x)^\beta a(t)\|_{L^p(\R^n)}\\
\le 
C_N\sum_{|\alpha|+3|\beta|\le 3N}\|(\eps\partial_\xi)^\alpha(\eps\partial_x)^\beta a(0)\|_{L^p(\R^n)}.
\end{multline}
Additionally, for all $N\in\Z_+$ and $T>0$ there is $C_{N,T}$ so that
\begin{equation}\label{eqn:globalest}
\sup_{T\le t\le \infty}\,\sum_{|\alpha|+3|\beta|\le N}
\|(\eps\partial_\xi)^\alpha(\eps\partial_x)^\beta a(t)\|_{L^p(\R^n)}
\le 
C_{N,T}\|a(0)\|_{L^p(\R^n)},
\end{equation}
where $C_{N,T}=\bigO(T^{-\frac 12 N})$ as $T\rightarrow 0^+$.
\end{theorem}

Estimate \eqref{eqn:globalest} captures the smoothing effect of $e^{tQ}$ for $t>0$.
Estimate \eqref{eqn:localest} is the result of the parametrix mapping properties, where the counting of $\alpha$ and $\beta$ reflects the relative order of $\partial_x$ versus $\partial_\xi$ in the nonisotropic subelliptic calculus associated to $Q$.

For simplicity, in proving Theorem \eqref{thm:main} we consider data $a(0)\in\S(\R^{2n})$; the a priori estimates will allow the extension of Theorem \ref{thm:main} to spaces of functions with $L^p$ bounds on a finite number of derivatatives. We start with the following two results.

\medskip

\noindent{\bf Maximum Principal}:
$\displaystyle
\sup_{t>0}\sup_{x,\xi} a(t,x,\xi)\le \sup_{x,\xi} a(0,x,\xi).
$

\medskip

\noindent{\bf Mass Conservation}: for $t>0$,
$\displaystyle
\;\int_{\R^{2n}} a(t,x,\xi)\,dx\,d\xi=\int_{\R^{2n}} a(0,x,\xi)\,dx\,d\xi.
$

\medskip

Both of these follow as for the standard heat equation; for completeness we include the proofs here.
For the maximum principle, it suffices to work on a strip $t\in [0,T]$ for $T<\infty$. Given $c>0$ consider the function $a_c=a-c t$, so $(\partial_t-Q)a_c=-c<0$. At a local maximum of $a_c$ over $0\le t\le T$ we must have $\nabla_x a=0=\nabla_\xi a$, and $X_j^2 a\le 0$ for $j\ge 1$. Additionally $\partial_t a\ge 0$ if $t\ne 0$. This leads to a contradiction if $t\ne 0$. Thus $a_c$ attains it maximum at $t=0$. Since this holds for all $c>0$, it follows that $a$ attains its maximum at $t=0$.

Mass conservation follows by writing
$$
\partial_t \int a\,dx\,d\xi=\int Q a\,dx\,d\xi.
$$
The first order terms in $Q$ are divergence free, and the second order terms have constant coefficients, hence we may integrate by parts, using decay of $a$ in $(x,\xi)$, to see that the right hand side vanishes.

Since the equation is linear and real, the maximum principal shows that 
$$
\|e^{tQ}a(0)\|_{L^\infty(\R^{2n})}\le \|a(0)\|_{L^\infty(\R^{2n})}.
$$
The same bound holds for the transpose of $Q$, so by duality this also holds for the $L^1$ norm.
Alternatively, the $L^1$ contractivity follows by positivity together with mass conservation. We conclude by interpolation that
\begin{equation}\label{eqn:Lpbound}
\|e^{tQ}a(0)\|_{L^p(\R^{2n})}\le \|a(0)\|_{L^p(\R^{2n})}\quad\text{for all}\;\;1\le p\le\infty.
\end{equation}

\begin{proof}[Proof of Theorem \ref{thm:main}]

Given $N$, let $T_K$ be a right inverse for $\partial_t-Q$ as in \eqref{eqn:rightinverse} such that the remainder is of order $-3N$, and let $g=T_K(\delta\otimes a(0))$. 
Since $K(v,y)-K_{0,\eps}(y)$ is of negative order, and $K_{0,\eps}(y_0,\cdot)\rightarrow\delta(\cdot)$ as $y_0\rightarrow 0+$, by \eqref{eqn:intbound} we have that
$\lim_{t\rightarrow 0+}g(t,\cdot)=a(0)$ in the $L^p(\R^{2n})$ norm.

Recall that $X'=B^T\partial_\xi$ and $X''=B^T\partial_x$ with $\det(B)\ne 0$. The coefficients of $K$ are continuous in $y_0$ and smooth in $v',v''$.
By Theorem \ref{thm:opbound} we deduce that, for every $m\in\Z_+$,
\begin{multline}\label{eqn:vbound}
\sup_{t>0}\sum_{|\alpha'|+3|\alpha''|\le 3m}\bigl\|(\eps \partial_\xi)^{\alpha'}(\eps \partial_x)^{\alpha''}g(t,\cdot)\bigr\|_{L^p}
\\
\le
C_N\sum_{|\alpha'|+3|\alpha''|\le 3m}\bigl\|(\eps \partial_\xi)^{\alpha'}(\eps \partial_x)^{\alpha''}a(0)\bigr\|_{L^p}.
\end{multline}
Additionally, $(\partial_t-Q)g=R_{3N+1}(\delta\otimes a(0))$ satisfies
\begin{equation}\label{eqn:Lvbound}
\sup_{t>0}\sum_{|\alpha'|+3|\alpha''|\le 3N}\bigl\|(\eps \partial_\xi)^{\alpha'}(\eps \partial_x)^{\alpha''}(\partial_t-Q)g(t,\cdot)\bigr\|_{L^p}
\le
C_N\|a(0)\|_{L^p}.
\end{equation}

Let $h=e^{tQ}a(0)-g$, so $(\partial_t-Q)h=-(\partial_t-Q)g$ and $h=0$ for $t\le 0$.
To control $h$, let $K$ be the left parametrix constructed in Theorem \ref{thm:epsleftinverse} for $\partial_t-Q$ with remainder $R_{3N+1}$ of order $-3N$, and write
\begin{multline*}
h(w)=-\int K(\eps,w,-y)\bigl(\partial_t g-Q g\bigr)(\exp_w(y))\,\chi(y)\,dy
\\
+\int R_{3N+1}(\eps,w,-y)h(\exp_w(y))\,dy.
\end{multline*}
By \eqref{eqn:Lpbound} and \eqref{eqn:vbound} with $m=0$ we have $\|h(t,\cdot)\|_{L^p}\le C\|a(0)\|_{L^p}$ for $t\ge 0$.
If $|\alpha'|+3|\alpha''|\le N$, and $w=(t,x,\xi)$, Theorem \ref{thm:opbound} with $m=0$ then yields
\begin{multline*}
\sup_{t>0}\,\bigl\|(\eps\partial_\xi)^{\alpha'}(\eps\partial_x)^{\alpha''} \int R_{3N+1}(\eps,w,-y)h(\exp_w(y))\,dy\,\bigr\|_{L^p(dx\,d\xi)}\\
\le C\|a(0)\|_{L^p},
\end{multline*}
where we use that the kernel of $R_{3N+1}$ is supported in $|y|<1$ so uniform in time $L^p$ bounds on $h$ control the integral over $y$. To handle the first term in $h$ we use \eqref{eqn:Lvbound} and Theorem \ref{thm:opbound} with $m=3N$. Together we conclude that
$$
\sup_{t>0}\sum_{|\alpha'|+3|\alpha''|\le 3N}\bigl\|(\eps\partial_\xi)^{\alpha'}(\eps\partial_x)^{\alpha''}h(t,\cdot)\bigr\|_{L^p(dx\,d\xi)}\le C\|a(0)\|_{L^p}.
$$
Together with \eqref{eqn:vbound} this establishes the bounds of \eqref{eqn:localest} on $u$.

To establish \eqref{eqn:globalest} we follow the same steps, and note that the derivatives up to order $3N$ of $h$ satisfy uniform $L^p$ bounds for $t\ge 0$. Additionally, $g$ is supported in $t\le 1$ and satisfies the bounds of \eqref{eqn:globalest} by Corollary \ref{cor:operbound}, noting that Lemma \ref{lem:operbound} is invariant under transpose.
\end{proof}


\section{Applications to classical-quantum correspondence}\label{sec:classquant}

We show here how Theorem \ref{thm:Li} and Theorem \ref{thm:gaz5} follow from Theorem \ref{thm:main} and results from \cite{Li} and \cite{gaz5}. We consider real valued $V$ satisfying \eqref{eqn:Vbounds}, and complex valued $\ell_j$ satisfying \eqref{eqn:cholesky} with non-degenerate $B$, and define $Q$ by \eqref{eqn:Qdef}. Recall that $S^{L^q}_\rho$ is defined by the seminorms
$$
\|a\|_{N,S^{L^q}_\rho}=\sum_{|\alpha|\le N} h^{\rho |\alpha|-\frac nq}\|\partial_{x,\xi}^\alpha a\|_{L^q}.
$$
We first consider Theorem \ref{thm:Li}, which concerns $q=1$ and the trace-class norm. As noted in Lemma 2.1 of \cite{Li}, the connection to the trace norm is given by Theorem 9.3 of \cite{Dim} and (C.3.1) of \cite{e-z},
\begin{equation}\label{eqn:L1trace}
\|\Op_h^\w(a)\|_{\L^1}\le C_n\|a\|_{2n+1,S^{L^1}_{1/2}}
\end{equation}

\begin{proof}[Proof of Theorem \ref{thm:Li}]
The proof follows that of Theorem 2 of \cite{Li}. We include the outline of the proof for convenience. Using $A(0)=\Op_h^\w(a(0))$ one writes
\begin{align*}
A(t)-\Op_h^\w(a(t))&=\int_0^te^{(t-s)\L}(\L-\partial_s)\Op_h^\w(a(s))\,ds\\
&=\int_0^te^{(t-s)\L}\Bigl(\L\Op_h^\w(a(s))-\Op_h^\w(Qa(s))\Bigr)\,ds.
\end{align*}
Since $e^{t\L}$ is a contraction on the space of trace class operators for $t\ge 0$, it suffices to verify that, uniformly for all $s\in [0,\infty)$,
$$
\Bigl\|\L\Op_h^\w(a(s))-\Op_h^\w(Qa(s))\Bigr\|_{\L^1}\le C h^{\frac 12}.
$$
Since $\gamma>0$ is fixed, we have $\eps=\sqrt{\gamma h/2}\approx h^{\frac 12}$. Theorem \ref{thm:main} thus implies
$$
\sup_{s\ge 0}\|a(s)\|_{N,S_{1/2}^{L^1}}\le C_N\,\|a(0)\|_{N,S_{1/2}^{L^1}},\quad \forall N\ge 0.
$$
The jump operators $L_j$ are linear functions of $x$, so the composition rules for the semi-classical Weyl calculus imply that, with $P=p^\w(x,hD)$,
\begin{equation}\label{eqn:comp}
\L\Op_h^\w(a(s))-\Op_h^\w(Qa(s))=\frac{i}{h}[P,\Op_h^\w(a(s))]-\Op_h^\w(H_pa(s)).
\end{equation}
The proof is thus concluded by applying composition results from \cite{e-z} as in Lemma 2.3 of \cite{Li} to obtain, for $N$ depending on $n$,
$$
\Bigl\|\frac{i}{h}[P,\Op_h^\w(a(s))]-\Op_h^\w(H_pa(s))\Bigr\|_{2n+1,S_{1/2}^{L^1}}\le 
C h^{\frac 12}\|a(s)\|_{N,S_{1/2}^{L^1}},
$$
which together with \eqref{eqn:L1trace} implies the desired estimate.
\end{proof}

Theorem \ref{thm:gaz5} extends results of \cite{gaz5} concerning estimates in the Hilbert-Schmidt norm, and relate to our estimates for $q=2$. The connection of $S_{\rho}^{L^2}$ to the Hilbert-Schmidt norm is given by 
\begin{equation}\label{eqn:L2trace}
\|\Op_h^\w(a(t))\|_{\L^2}=(2\pi h)^{-\frac n2}\|a(t)\|_{L^2}=(2\pi)^{-\frac n2}\|a(t)\|_{0,S^{L^2}_\rho}.
\end{equation}

\begin{proof}[Proof of Theorem \ref{thm:gaz5}]
The proof is similar to that of Theorem \ref{thm:Li}.
The semigroup $e^{t\L}$ is a contraction on Hilbert-Schmidt operators for $t\ge 0$ by Proposition 4.6 of \cite{gaz5}, as $M=0$ in our case, so it suffices to verify that, uniformly for all $s\in [0,\infty)$,
$$
\Bigl\|\L\Op_h^\w(a(s))-\Op_h^\w(Qa(s))\Bigr\|_{\L^2}\le C h^{2-3\rho}.
$$
Since $1\ge\gamma\ge h^{2\rho-1}$ we have $h^{\frac 12}\ge\eps\ge h^{\rho}$. Theorem \ref{thm:main} thus implies
$$
\sup_{s\ge 0}\|a(s)\|_{N,S_\rho^{L^2}}\le C_N\,\|a(0)\|_{N,S_\rho^{L^2}},\quad \forall N\ge 0.
$$
By \eqref{eqn:comp} and \eqref{eqn:L2trace} the theorem follows from showing that
$$
\Bigl\|\frac{i}{h}[P,\Op_h^\w(a(s))]-\Op_h^\w(H_pa(s))\Bigr\|_{0,S_\rho^{L^2}}\le 
C h^{2-3\rho}\|a(s)\|_{N,S_\rho^{L^2}}
$$
for some $N$ depending on $n$, which holds by Lemma 2.2 of \cite{gaz5}.
\end{proof}

\section*{Acknowledgements} I would like to thank Maciej Zworski for bringing the topic of this paper to my attention, and for his encouragement in pursuing the results obtained herein. I would also like to thank Kevin Li for helpful discussions on the details of his paper \cite{Li}.



\nocite{*}
\bibliographystyle{plainurl}
\bibliography{FokkerPlanck}


\end{document}